\documentclass{amsart}
\usepackage{graphicx} 

\usepackage{preamble}
\DeclareMathOperator{\Sdepth}{Sdepth}
\DeclareMathOperator{\depth}{depth}
\DeclareMathOperator{\coker}{coker}
\DeclareMathOperator{\Ann}{Ann}
\DeclareMathOperator{\Ass}{Ass}
\DeclareMathOperator{\Ext}{Ext}
\DeclareMathOperator{\Rad}{Rad}
\title[Serre depth, generalized CM, and LLC]{$S_r$ properties of generalized Cohen--Macaulay modules and rings having liftable local cohomology via Serre depth}

\author{Haocheng Cai}

\begin{document}

\begin{abstract}
We study Serre depth via Matlis duals of local cohomology modules. We relate the Serre depth of a module to that of a quotient by a regular element, characterize the $S_r$ property for generalized Cohen--Macaulay modules in terms of ordinary depth, and show that the $S_r$ property descends to reduced structures under liftable local cohomology hypotheses.
\end{abstract}


\maketitle
\section{Introduction}
Cohen--Macaulay modules are important in commutative algebra and algebraic geometry. A well-known generalization is given by Serre's $S_r$ conditions. While checking Cohen--Macaulayness is relatively easy, checking $S_r$ properties is more laborious. The notion of Serre depth (Definition \ref{def:serre}) was recently introduced in \cite{PPTY,PPTY2}. However, several of its key properties have already appeared in \cite{Schenzel}, and were probably known to experts even earlier. The relation between Serre depth and the $S_r$ property can be exhibited using techniques in \cite{CM}; for details, see \cite{FicarraSerreDepth}.

An important obstacle for using the Cohen--Macaulay property is that the reduced structure of a Cohen--Macaulay scheme may no longer be Cohen--Macaulay; see Example \ref{ex:quotient}. In this paper, we show that this cannot happen under liftable local cohomology hypotheses; see Theorem \ref{thmintro:S_r} and Corollary \ref{cor:CM}. Liftable local cohomology was introduced and studied in \cite{KollarKovacsDeformLC} to prove the flatness of the cohomology sheaves of the dualizing complex. A very similar concept, called cohomologically full rings, is studied in \cite{LLC}. It was shown to be equivalent to $F$-fullness in characteristic $p$, and it includes many important families of singularities in characteristic 0, such as Cohen--Macaulay rings and rings with Du Bois singularities. Our proof for rings having liftable cohomology works also for cohomologically full rings.

\begin{theorem}\label{thmintro:S_r}
    Let $(T,n)$ be a local ring over a Noetherian ring $A$ with liftable local cohomology, and let $(R,m)$ be a local $A$-algebra such that $R/I\cong T$ with $I$ a nilpotent ideal. Assume in addition that $R$ is a homomorphic image of a Gorenstein ring and is equidimensional. If $R$ is $S_r$ over $R$, then $T$ is $S_r$ over $T$.
\end{theorem}

Another result in this paper is that we can read off an $S_r$ property of a generalized Cohen--Macaulay module (Definition \ref{def:gCM}) from its depth:

\begin{theorem}\label{introthm:gCM}
    Let $M$ be a generalized Cohen--Macaulay module over a local ring $(R,m,k)$ with $\dim(M)=n$. Then $M$ is $S_r$ where $r=\depth(M)$. Conversely, assume in addition that $M$ is equidimensional and $R$ is a homomorphic image of a Gorenstein ring. If $M$ is $S_r$ for some $r\leq n$, then $\depth(M)\geq r$.
\end{theorem}

We also give a formula relating the Serre depth of $M$, denoted by $\Sdepth^r(M)$ (see Definition \ref{def:serre}) and that of $M/xM$ for a nonzerodivisor $x$; see Theorem \ref{thm:bertini}. The content of the theorem is that we interpret the inequalities in the definition of Serre depth as more detailed conditions $C(i)$; see Definition \ref{def:C(i)}. This interpretation is obtained by introducing a key dimension parameter $d_x(N)$; see Definition \ref{def:d_x} and its use in Lemma \ref{lem:dim} and Theorem \ref{thm:dim}.

\begin{theorem}\label{introthm:bertini}
    Let $(R,m,k)$ be a Noetherian local ring, $M$ a finitely generated $R$-module of dimension $n$, and $x$ an $M$-regular element in $m$. Let $\Sdepth^r(M)=s$. Then:
    \begin{enumerate}
        \item if none of the conditions $C(i)$ hold, then $\Sdepth^{r+1}(M/xM)=s-1$;
        \item otherwise, $\Sdepth^{r+1}(M/xM)$ is the smallest index $i$ such that $C(i)$ holds.
    \end{enumerate}
\end{theorem}
The condition $C(i)$ in the theorem statement above is defined in Definition \ref{def:C(i)}.

The concept of Serre depth was developed in detail in \cite{FicarraSerreDepth} and appeared previously in \cite{MutaTeraiSerreDepth,PPTY,PPTY2}. It builds on information from the Matlis duals of local cohomology modules over a local Noetherian ring or a graded $k$-algebra. It stratifies Serre's $S_r$ condition. In particular, under mild conditions, the $r$th Serre depth reaches the maximum value if and only if the module is $S_r$. The idea of Serre depth originates from the characterization of Cohen--Macaulayness by the vanishing of local cohomology. A finitely generated module $M$ over a Noetherian local ring $R$ is Cohen--Macaulay if and only if $H^i_m(M)=0$ for all $0\leq i< \dim M$. We can ask whether the vanishing assumption on local cohomology can be weakened to obtain a characterization of $S_r$. One difficulty is that the local cohomology modules, if nonzero, are usually not finitely generated. Fortunately, they are Artinian, and we can take Matlis duals to obtain finitely generated modules. We then measure their Krull dimensions. The existence of local cohomology modules whose Matlis duals have large Krull dimension is an obstruction to the module being $S_r$.

The structure of this paper is as follows. Section \ref{notation} introduces the definition of Serre depth (Definition \ref{def:serre}) and recalls results from the literature used in later sections. In Section \ref{sec:quotient}, we investigate Matlis duals of local cohomology modules and give numerical relations between $\Sdepth^r(M)$ and $\Sdepth^{r+1}(M/xM)$, where $x$ is a nonzerodivisor on $M$. In Section \ref{sec:generalizedCM}, we characterize the $S_r$ property of generalized Cohen--Macaulay modules. In Section \ref{sec:LLC}, we prove that for rings having liftable local cohomology, Serre's $S_r$ property propagates to the reduced structure.

\section{Definitions and properties of Serre depth}\label{notation}
Throughout the present article, $(R,m,k)$ denotes a Noetherian local ring with maximal ideal $m$ and residue field $k$, $M$ is a finitely generated $R$-module, and $H^i_m(M)$ is the $i$th local cohomology of $M$ with respect to $m$. We write $E(k)$ for the injective hull of $k$, and abbreviate the Matlis dual functor $\Hom_R(\_,E(k))$ by $D(\_)$. The Matlis duals of local cohomology modules are the central objects of study for Serre depth.
\begin{definition}\label{def:serre}
    Let $r\geq 1$ be an integer. The \textbf{r-th Serre depth} of $M$ over a Noetherian local ring $(R,m,k)$, denoted by $\Sdepth^r_R(M)$, is defined to be \begin{equation}
        \Sdepth^r_R(M):=\min\{j:\dim DH^j_m(M)\geq j-r+1\}.
    \end{equation}
    The dimension here is the Krull dimension over $R$. We use the convention that the zero module has dimension $-\infty$. When the ring $R$ is understood, we write $\Sdepth^r(M)$ for $\Sdepth^r_R(M)$.
\end{definition}

$\Sdepth^r(M)$ is weakly decreasing and bounded by $\dim(M)$ and $\depth(M)$:
\begin{proposition}[{\protect\cite[Prop. 1.4]{FicarraSerreDepth}}]\label{prop:serre-bound}
    Let $M$ be a finitely generated module over a Noetherian local ring $(R,m)$. Then:
    \begin{enumerate}
        \item $\dim(M)\geq \Sdepth^1(M)\geq\cdots\geq \Sdepth^{r-1}(M)\geq \Sdepth^r(M)\geq\cdots$.
        \item $\Sdepth^r(M)=\depth(M)$ for every positive integer $r\geq \dim(M)$.
    \end{enumerate}
\end{proposition}

\begin{remark}\label{rem:top}
    The original proof of (1) in \cite[Prop 1.4]{FicarraSerreDepth} did not rule out the case that for all $j$, $\dim DH^j_m(M)< j-r+1$. However, the second proof in the remark immediately following it works using the theory of attached primes: the Matlis dual of the top local cohomology module has dimension of its support equal to the dimension of the module. That is, if $\dim(M)=n$, then $\dim DH^n_m(M)=n\geq n-r+1$ for all $r\geq 1$. Hence, when we take the minimum as in the definition \ref{def:serre}, it is always smaller or equal to $\dim(M)=n$.
\end{remark}
Analogously to the fact that $\depth(M)=\dim(M)$ if and only if $M$ is Cohen--Macaulay, when $\Sdepth^r(M)=\dim(M)$, we recover Serre's $S_r$ property.
\begin{definition}
    Let $(R,m)$ be a local ring and let $M$ be an $R$-module. We say that $M$ is \textbf{equidimensional} if all minimal primes of $\Ann(M)$ have the same dimension.
\end{definition}
\begin{theorem}[{\protect\cite[Theorem B]{FicarraSerreDepth}}]\label{thm:SerreDepth}
    If $\Sdepth^r(M)=\dim(M)$, then $M$ is $S_r$.
    The converse holds if we assume, in addition, that $R$ is a homomorphic image of a Gorenstein ring and $M$ is equidimensional.
\end{theorem}
\begin{remark}
    As mentioned in \cite{FicarraSerreDepth}, if $r\geq 2$ and $M$ is indecomposable or a quotient of $R$, then $M$ is equidimensional. Hence the converse of Theorem \ref{thm:SerreDepth} applies in these cases.
\end{remark}

\begin{example}\label{ex:equidim}
    The equidimensional assumption in the converse is essential. Let $M=M_1\oplus M_2$. Then $M$ is $S_r$ if and only if $M_1$ and $M_2$ are $S_r$. However, the definition of Serre depth tests the dimensions of the Matlis duals of local cohomology against the dimension of $M$, rather than against the dimensions of $M_1$ and $M_2$ separately.

    Here is a concrete example from \cite{FicarraSerreDepth}: let
    \[
        R=k[[x_1,x_2,x_3,x_4,x_5]],\quad
        M_1=k[[x_1,x_2,x_3]],\quad
        M_2=k[[x_4,x_5]],
    \]
    and set $M=M_1\oplus M_2$. Then $\Ann(M)=(x_4,x_5)\cap(x_1,x_2,x_3)$, and $M$ is not equidimensional because the minimal primes over $\Ann(M)$ are $(x_4,x_5)$ and $(x_1,x_2,x_3)$, which have different dimensions. The module $M$ is $S_2$ since $M_1$ and $M_2$ are $S_2$. On the other hand, we calculate the Serre depth. By \cite[3.5.7, 3.5.9]{CM}, $\dim DH^2_m(M_1)=0$, and by \cite[3.3.3, 3.5.9]{CM}, $\dim DH^2_m(M_2)=2$. Hence $\dim DH^2_m(M)=2\geq 2-2+1=1$. Again by Grothendieck's vanishing theorem, $DH^i_m(M)=0$ for $i=0,1$, and hence $\dim DH^i_m(M)=-\infty$ by our convention. Thus $\Sdepth^2(M)=2<3=\dim M$.

    Therefore $M$ satisfies $S_2$, but $\Sdepth^2(M)\neq\dim M$.
\end{example}

\section{Serre depth with respect to quotients by regular elements}\label{sec:quotient}
Given $x\in m$ which is a nonzerodivisor on $M$, we can form the short exact sequence $0\rightarrow M\xrightarrow{x} M\rightarrow M/xM\rightarrow 0$. Then we get a long exact sequence of local cohomology
\begin{align}
    0\rightarrow H^0_m(M)\rightarrow H^0_m(M)\rightarrow H^0_m(M/xM)\rightarrow H^1_m(M)\rightarrow \cdots\\
    \rightarrow H^i_m(M)\rightarrow H^i_m(M)\rightarrow H^i_m(M/xM)\rightarrow\cdots
\end{align}
Since taking Matlis duals is a contravariant exact functor on Artinian modules, applying it to the long exact sequence above gives another exact sequence:
\begin{equation}\label{les}
    \cdots \rightarrow DH^{i+1}_m(M)\xrightarrow{\phi} DH^{i+1}_m(M)\rightarrow DH^{i}_m(M/xM)\rightarrow DH^{i}_m(M)\xrightarrow{\theta} DH^{i}_m(M)\cdots
\end{equation}
In the definition of Serre depth, only the Krull dimension matters. Thus it is natural to deduce the dimension of $DH^i_m(M/xM)$ from the dimensions of $DH^{i+1}_m(M)$ and $DH^{i}_m(M)$ by using the short exact sequence
\begin{equation}
    0\rightarrow \coker(\phi)\rightarrow DH^i_m(M/xM)\rightarrow \ker(\theta)\rightarrow 0
\end{equation}
Another observation is that both $\phi$ and $\theta$ are just multiplication by $x$. Hence, we make the following definition:
\begin{definition}\label{def:d_x}
    Given a finitely generated module $N$ over a Noetherian ring $R$, and $x\in R$ a non-unit, we define $d_x(N)\in \mathbb Z\cup\{-\infty\}$ by
\begin{equation}
    d_x(N):= \max\{\dim(R/p):p\in \Ass(N)\cap V(x)\}.
\end{equation}
\end{definition}

\begin{remark}
    Notice that it is possible that $\Ass(N)\cap V(x)=\emptyset$, in which case we define $d_x(N)=-\infty$.
    This is consistent with the convention that the dimension of the zero module is $-\infty$.
\end{remark}

\begin{lemma}\label{lem:dim}
    Under the same assumptions as in Definition \ref{def:d_x}, consider the multiplication by $x$ map $\phi_x:N\rightarrow N$.
    \begin{enumerate}
        \item $\dim \coker(\phi_x)=\dim N/xN=\dim N-1$ if and only if $d_x(N)<\dim N$; equivalently, $\dim \coker(\phi_x)=\dim N$ if and only if $d_x(N)=\dim N$.
        \item $\dim(\ker(\phi_x))=d_x(N)$.
    \end{enumerate}
\end{lemma}
\begin{proof}
    (1) Assume first that $d_x(N)=\dim N=n$. Choose $p\in \Ass(N)\cap V(x)$ such that $\dim R/p=n$. Then $p$ is a minimal prime over $\Ann(N)$. Since $x\in p$, the prime $p$ remains in $\Supp(N/xN)$. Hence $\dim(N/xN)\geq n$, and the reverse inequality is automatic. Thus $\dim(N/xN)=n$.

    Conversely, assume that $d_x(N)<\dim(N)=n$. If $\dim(N/xN)=n$, then there is a prime $p\in \Supp(N/xN)$ with $\dim R/p=n$. Choose a minimal prime $q$ of $N$ with $q\subseteq p$. Since $\dim R/q\leq n$ and $\dim R/p=n$, we must have $\dim R/q=n$ and hence $q=p$. Thus $p$ is a top-dimensional associated prime of $N$. Since $x\in p$, this contradicts $d_x(N)<n$. Therefore $\dim(N/xN)<n$, and the principal ideal theorem gives $\dim(N/xN)=n-1$.

    (2) We have
    \[
        \Ass(\ker(\phi_x))=\Ass(N)\cap V(x).
    \]
    Indeed, if $p\in \Ass(\ker(\phi_x))$, then $p=\Ann(t)$ for some $t\in \ker(\phi_x)$. Hence $xt=0$ and $x\in \Ann(t)=p$. Also, since $\ker(\phi_x)\subseteq N$, we have $p\in\Ass(N)$. Conversely, if $p\in\Ass(N)\cap V(x)$, then $p=\Ann(t)$ for some $t\in N$. Since $x\in p$, we have $xt=0$, so $t\in\ker(\phi_x)$ and $p\in\Ass(\ker(\phi_x))$. Taking dimensions gives $\dim(\ker(\phi_x))=d_x(N)$.
\end{proof}

\begin{theorem}\label{thm:dim}
    Set $A_j=DH^j_m(M)$. Then
    \[
    \begin{aligned}
    \dim DH^i_m(M/xM)
    ={}& \begin{cases}
    \max\{\dim A_{i+1},d_x(A_i)\},
    &\text{if }d_x(A_{i+1})=\dim A_{i+1}, \\
    \max\{\dim A_{i+1}-1,d_x(A_i)\},
    &\text{if }d_x(A_{i+1})<\dim A_{i+1}.
    \end{cases}
    \end{aligned}
    \]
\end{theorem}
   \begin{proof}
       We use the short exact sequence
    \begin{equation}\label{ses}
    0\rightarrow \coker(\phi)\rightarrow DH^i_m(M/xM)\rightarrow \ker(\theta)\rightarrow 0
    \end{equation}
    In a short exact sequence, the dimension of the middle term is the maximum of the dimensions of the two outer terms.
    We split into two cases. If
    \[
        d_x(DH^{i+1}_m(M))=\dim DH^{i+1}_m(M),
    \]
    then Lemma \ref{lem:dim}(1) gives $\dim \coker(\phi)=\dim DH^{i+1}_m(M)$. If
    \[
        d_x(DH^{i+1}_m(M))<\dim DH^{i+1}_m(M),
    \]
    then Lemma \ref{lem:dim}(1) gives $\dim \coker(\phi)=\dim DH^{i+1}_m(M)-1$. In both cases, Lemma \ref{lem:dim}(2) gives $\dim \ker(\theta)=d_x(DH^i_m(M))$. The theorem follows by combining these computations.
   \end{proof}
\begin{remark}
    An interesting intuition we can get from Theorem \ref{thm:dim} is that a larger $d_x(DH^i_m(M))$ would imply a larger local cohomology module of $M/xM$.
\end{remark}

Theorem \ref{thm:dim} is very powerful for the study of Serre depth because if we know $\dim DH^i_m(M/xM)$ for all $i$, then we can easily determine $\Sdepth^r(M)$ for all $r$.
If we specialize to particular indices $r_1,r_2$, we can even relate $\Sdepth^{r_1}(M)$ with $\Sdepth^{r_2}(M/xM)$ with conditions on the $d_x(DH^i_m(M))$ and $r_1,r_2$.
The most interesting case in terms of the use of $d_x(DH^i_m(M))$ is presented in the rest of the section.

We now try to understand $\Sdepth^{r+1}(M/xM)$ from the local cohomology of $M$. It is not hard to see using Nakayama's lemma that $\Sdepth^{r+1}(M/xM)\leq \Sdepth^r(M)-1=s-1$ (see proof of Theorem \ref{thm:bertini}).
However, by the definition of Serre depth, the inequality can be strict if for some index $i< s-1$, $\dim DH^{i}_m(M/xM)\geq i-(r+1)+1=i-r$.
Hence we want to take a closer look at the long exact sequence of the dual of the local cohomology (see below), and to address the case of ``the largest possible $\dim DH^{i}_m(M/xM)$".
One natural idea is to force, for two adjacent maps, one mapping to $DH^{i+1}_m(M)$ and the other one from $DH^{i}_m(M)$, to have the largest possible cokernel and kernel.
They have already been studied in Lemma \ref{lem:dim} and Theorem \ref{thm:dim}. The only thing left to do is to patch together the data.

\begin{lemma}\label{lem:iff}
     Let $(R,m,k)$ be a Noetherian local ring, $M$ a finitely generated $R$-module of dimension $n$, and $x\in R$ a nonzerodivisor on $M$. Let $s=\Sdepth^r(M)$. For each $i$ satisfying $\depth(M)-1\leq i< s-1$, we have $\dim DH^{i}_m(M/xM)\geq i-r$ if and only if one of the following conditions holds:
    \begin{enumerate}
        \item $\dim DH^{i+1}_m(M)=i-r+1$;
        \item $\dim DH^{i+1}_m(M)=i-r$ and $d_x(DH^{i+1}_m(M))=\dim DH^{i+1}_m(M)$;
        \item $\dim DH^{i}_m(M)=i-r$ and $d_x(DH^{i}_m(M))=\dim DH^{i}_m(M)$.
    \end{enumerate}

\end{lemma}
\begin{proof}
    We apply Theorem \ref{thm:dim} repeatedly. If (1) holds, then even the smaller case in Theorem \ref{thm:dim} guarantees that $\dim DH^{i}_m(M/xM)\geq i-r$. If (2) holds, then $d_x(DH^{i+1}_m(M))=\dim DH^{i+1}_m(M)$, so the first case of Theorem \ref{thm:dim} gives $\dim DH^{i}_m(M/xM)\geq \dim(DH^{i+1}_m(M))=i-r$. If (3) holds, then Theorem \ref{thm:dim} gives $\dim DH^{i}_m(M/xM)\geq d_x(DH^{i}_m(M))=i-r$.

    Conversely, since $s=\Sdepth^r(M)$ and $i<s-1$, the definition of Serre depth gives $\dim DH^{i+1}_m(M)\leq i+1-r$ and $\dim DH^{i}_m(M)\leq i-r$. Enumerating all cases in Theorem \ref{thm:dim} in which $\dim DH^{i}_m(M/xM)\geq i-r$, we are left with precisely conditions (1), (2), and (3).
\end{proof}
\begin{definition}\label{def:C(i)}
    Under the same assumptions on $(R,m)$, $M$, and $x\in R$ as in Lemma \ref{lem:iff}, we say that \textbf{condition $C(i)$ holds} if one of the conditions (1), (2), or (3) in Lemma \ref{lem:iff} holds.
\end{definition}
\begin{remark}
    Notice that $C(i)$ is defined only for $i$ such that $\depth(M)-1\leq i< s-1$. By Lemma \ref{lem:iff}, $C(i)$ holds if and only if $\dim DH^{i}_m(M/xM)\geq i-r$.
\end{remark}

\begin{theorem}\label{thm:bertini}
    With the same notation as in Lemma \ref{lem:iff}, we have:
    \begin{enumerate}
        \item If $C(i)$ does not hold for any $i$ where $\depth(M)-1\leq i< s-1$,
        then $\Sdepth^{r+1}(M/xM)=s-1$.
        \item Otherwise $\Sdepth^{r+1}(M/xM)$ is equal to the smallest index $i$ such that $C(i)$ holds.
    \end{enumerate}
\end{theorem}
\begin{proof}
    (1) By the equivalence proved in Lemma \ref{lem:iff} and the definition of Serre depth, it is clear that $\Sdepth^{r+1}(M/xM)\geq s-1$. To check that equality holds, we need to show $\dim DH^{s-1}_m(M/xM)\geq (s-1)-(r+1)+1=s-r-1$.
    Since $\Sdepth^r(M)=s$, we have $\dim DH^s_m(M)\geq s-r+1$. Applying Theorem \ref{thm:dim} with $i=s-1$, we have $\dim DH^{s-1}_m(M/xM)\geq \dim DH^s_m(M)-1\geq s-r$, as needed.
    For (2), again it follows from the definition of Serre depth and the equivalence in Lemma \ref{lem:iff}.
\end{proof}

\begin{remark}
    A related result was proved in \cite[Corollary 1.8]{FicarraSerreDepth}, namely
    \[
        \Sdepth^{r+1}(M)\leq \Sdepth^r(M/xM)\leq \Sdepth^r(M).
    \]
    It was then used by induction to generalize to quotients by regular sequences. That result focuses on bounds for Serre depth that correlate with the length of the regular sequence, while Theorem \ref{thm:bertini} focuses on finding the exact value of the Serre depth.
\end{remark}

\begin{remark}
    Note that $d_x(DH^i_m(M))$ is usually not easy to calculate. To find it, we need to understand the associated primes of $DH^i_m(M)$. However, there are cases that we can compute. Here is a class of examples from \cite[Theorem 3.4]{M}.
\end{remark}
\begin{example}
    Let $R=k[x_1,\ldots,x_n]_{(x_1,\ldots,x_n)}$, let $m$ be its maximal ideal, and let $B$ be a reduced monomial ideal. Then $\Ass(\Ext^i_R(R/B,R))$ is characterized using only the Betti numbers of $B^\vee$, the Alexander dual of $B$.
    Note that by local duality, $\Ext^i_R(R/B,R)$ is exactly the Matlis dual of $H^{n-i}_m(R/B)$. Hence, we may compute $d_x(DH^i_m(R/B))$ by combinatorial methods.
\end{example}

\section{\texorpdfstring{$S_r$}{S-r} property of generalized Cohen--Macaulay modules}\label{sec:generalizedCM}

\begin{definition}[{\protect\cite{NgoGMCM}}]\label{def:gCM}
    Let $(R,m,k)$ be a Noetherian local ring, and let $M$ be a finitely generated $R$-module. We say that $M$ is a \textbf{generalized Cohen--Macaulay module} if
    \begin{equation}
        \ell(H^i_m(M))<\infty
    \end{equation}
    for all $0\leq i<\dim M$, where $\ell$ denotes length.
\end{definition}

\begin{theorem}\label{thm:gCM}
    Let $M$ be a generalized Cohen--Macaulay module over $(R,m,k)$ with $\dim(M)=n$. Then $M$ is $S_r$, where $r=\depth(M)$.
    Conversely, assume in addition that $M$ is equidimensional and $R$ is a homomorphic image of a Gorenstein ring. If $M$ is $S_r$ for some $r\leq n$, then $\depth(M)\geq r$.
\end{theorem}
\begin{proof}
    We start with the first statement and assume that $r=\depth(M)$. By Theorem \ref{thm:SerreDepth}, it suffices to show that $\Sdepth^r(M)=n$. That is, $\dim DH^i_m(M)< i-r+1$ for all $i< n$, by Definition \ref{def:serre}.
    As $M$ is generalized Cohen--Macaulay, we have
    \begin{equation}
        \ell(H^i_m(M))<\infty
    \end{equation}
    for all $0\leq i<\dim M$.
    Taking Matlis duals sends finite length modules to finite length modules, hence
    \begin{equation}
        \ell(DH^i_m(M))<\infty.
    \end{equation}
    By the convention that the zero module has dimension $-\infty$ and the fact that a nonzero finite length module has dimension $0$,
    \[
     \dim(DH^i_m(M))=
     \begin{cases}
    0 &\text{if }H^i_m(M)\neq 0,\\
    -\infty &\text{if }H^i_m(M)=0.
    \end{cases}
    \]
    If $r\leq i<n$, then $\dim DH^i_m(M)\leq 0<i-r+1$. If $0\leq i<r$, then $H^i_m(M)=0$ because $r=\depth(M)$, and hence $\dim DH^i_m(M)=-\infty<i-r+1$. Combining the two cases, we have $\dim DH^i_m(M)< i-r+1$ for all $i<n$.

    For the converse statement, suppose that $M$ is $S_r$ and that $r\leq n$. Again by Theorem \ref{thm:SerreDepth}, we have $\Sdepth^r(M)=n$. For every $0\leq i<r$, we have $i-r+1\leq 0$ and, since $i<n$, the equality $\Sdepth^r(M)=n$ implies $\dim DH^i_m(M)<i-r+1$. Hence $DH^i_m(M)=0$, so $H^i_m(M)=0$, for all $0\leq i<r$. Therefore $\depth(M)\geq r$.
\end{proof}

\begin{remark}\label{rem:depth}
     The converse does not hold in general; that is, it is not always true that $M$ is $S_r$ for $r=\depth(M)$. See Example \ref{ex:equidim}.
\end{remark}

\section{Serre depth and liftable local cohomology}\label{sec:LLC}
\begin{definition}[{\protect\cite[Definition 1.1]{KollarKovacsDeformLC}}]
    Let $A$ be a Noetherian ring, and let $(T,n)$ be a Noetherian local $A$-algebra. We say that $T$ \textbf{has liftable local cohomology over $A$} if, for any Noetherian local $A$-algebra $(R,m)$ and nilpotent ideal $I \subseteq R$ such that $R/I \cong T$, the natural morphism on local cohomology
    \begin{equation}
        H^i_m(R)\rightarrow H^i_n(T)
    \end{equation}
    is surjective.
    For a scheme $X$ over $\Spec A$, we say that $X$ has \textbf{liftable local cohomology over $A$} if all local rings of $X$ have liftable local cohomology over $A$.
\end{definition}

\begin{proposition}\label{prop:LLC}
    Let $(T,n)$ have liftable local cohomology over a Noetherian ring $A$, and let $(R,m)$ be an $A$-algebra such that $R/I\cong T$ with $I$ a nilpotent ideal. Then, as $R$-modules,
    \[
        \Sdepth^r_R(R)\leq \Sdepth^r_R(T).
    \]
\end{proposition}
\begin{proof}
    By the definition of liftable local cohomology, we have surjections for all $i$:
    \begin{equation}
        H^i_m(R)\twoheadrightarrow H^i_n(T)
    \end{equation}
    We can also view $T$ as an $R$-module, and $H^i_m(T)\cong H^i_n(T)$ as $R$-modules; see \cite[Theorem 4.2.1]{BrodmannSharpLC}. Hence we have $R$-module surjections
    \begin{equation}
        H^i_m(R)\twoheadrightarrow H^i_m(T)
    \end{equation}
    Applying the Matlis dual $D:=\Hom_R(\_,E_R(k))$, we have
    \begin{equation}
        DH^i_m(T)\hookrightarrow DH^i_m(R)
    \end{equation}
    Hence
    \[
        \dim_R DH^i_m(T)\leq \dim_R DH^i_m(R).
    \]
    By Definition \ref{def:serre}, $\Sdepth^r_R(R)\leq \Sdepth^r_R(T)$.
\end{proof}

\begin{theorem}\label{thm:S_r}
    Let $(T,n)$ have liftable local cohomology over a Noetherian ring $A$, and let $(R,m)$ be an $A$-algebra such that $R/I\cong T$ with $I$ a nilpotent ideal. Assume in addition that $R$ is a homomorphic image of a Gorenstein ring and is equidimensional. If $R$ is $S_r$ over $R$, then $T$ is $S_r$ over $T$.
\end{theorem}
\begin{proof}
    Notice that $T$ is $S_r$ over $T$ if and only if $T$ is $S_r$ over $R$: primes $p\subseteq R$ containing $I$ correspond to primes $\bar p\subseteq T$, and localization as an $R$-module agrees with localization as a $T$-module. Regular sequences behave well with respect to quotienting by a nilpotent ideal. Hence
    \begin{equation}
        \depth_R(T_p)\geq \min\{\dim_R T_p,r\}
    \end{equation}
    is equivalent to
    \begin{equation}
        \depth_T(T_{\bar p})\geq \min\{\dim_T T_{\bar p},r\}.
    \end{equation}
    Thus it is enough to prove that $T$ is $S_r$ as an $R$-module. By Theorem \ref{thm:SerreDepth}, under our assumptions, $R$ being $S_r$ is equivalent to $\Sdepth^r_R(R)=\dim R$. Also, $\dim T=\dim R$ since $T\cong R/I$ with $I$ nilpotent. Combining this with Proposition \ref{prop:LLC}, we obtain $\dim R=\Sdepth^r_R(R)\leq \Sdepth^r_R(T)$.
    By Proposition \ref{prop:serre-bound}, we know that the Serre depth is bounded above by the dimension of the ring. Hence $\Sdepth^r_R(T)=\dim R=\dim T$. Again by Theorem \ref{thm:SerreDepth}, we obtain that $T$ satisfies the $S_r$ property over $R$. By the discussion at the start, we see that $T$ is $S_r$ over $T$.
\end{proof}

\begin{corollary}\label{cor:CM}
    Let $A$ be a Noetherian ring, and let $X$ be a locally Noetherian Cohen--Macaulay scheme over $\Spec A$. Assume that for every point $x\in X$, the local ring $\mathcal O_{X,x}$ is a homomorphic image of a Gorenstein ring and is equidimensional, and that the reduced local ring
    \[
        \mathcal O_{X_{\mathrm{red}},x}\cong (\mathcal O_{X,x})_{\mathrm{red}}
    \]
    has liftable local cohomology over $A$ (In particular, this applies if X is either Du Bois in char 0 or F-pure in char p, cf. \cite{LLC}). Then $X_{\mathrm{red}}$ is Cohen--Macaulay.
\end{corollary}
\begin{proof}
    For each $x\in X$, apply Theorem \ref{thm:S_r} to $R=\mathcal O_{X,x}$ and $T=\mathcal O_{X_{\mathrm{red}},x}$. The kernel of the natural map $R\to T$ is the nilradical of $R$, which is nilpotent because $R$ is Noetherian. Since $X$ is Cohen--Macaulay, each $R$ is Cohen--Macaulay, hence $S_r$ for all $r$. The theorem therefore implies that each reduced local ring $T$ is $S_r$ for all $r$, equivalently Cohen--Macaulay. Thus $X_{\mathrm{red}}$ is Cohen--Macaulay.
\end{proof}

\begin{remark}
    In general, without the liftable local cohomology assumption, it is possible for a ring $R$ to be Cohen--Macaulay while its reduced structure is not Cohen--Macaulay. See the following example by Geyer \cite{G}.
\end{remark}

\begin{example}\label{ex:quotient}
    Let $k$ be a field of characteristic $2$, let
    \[
        S=k[x_1,x_2,x_3,x_4,x_5]_{(x_1,x_2,x_3,x_4,x_5)},
    \]
    let $f=x_5^2-x_1x_3^2$ and $g=x_4^2-x_1x_2^2$, and set $R=S/(f,g)$. Since $f$ and $g$ do not have common factors, they form a regular sequence. Hence $R$ is Cohen--Macaulay. Notice that
    \[
        (x_4x_5-x_1x_2x_3)^2
        =x_4^2x_5^2-x_1^2x_2^2x_3^2
        =(x_1x_3^2)(x_1x_2^2)-x_1^2x_2^2x_3^2=0
    \]
    in $R$, so $R$ is not reduced. By Macaulay2, we have
    \[
        \Rad(f,g)=(x_5^2-x_1x_3^2,x_4^2-x_1x_2^2,x_2x_5-x_3x_4,x_4x_5-x_1x_2x_3)=J,
    \]
    and $S/J$ is not Cohen--Macaulay.
\end{example}

When we take the quotient of a local ring by an arbitrary ideal instead of a nilpotent ideal as in the setting above, we still have good behavior with respect to Serre depth.
\begin{theorem} \label{thm:ringchange}
    Let $(T,n)$ and $(R,m)$ be local rings with $T\cong R/I$ for some ideal $I$, and let $M$ be a $T$-module. Then $\Sdepth^r_R(M)=\Sdepth^r_T(M)$.
\end{theorem}
\begin{proof}
    As in the proof of Proposition \ref{prop:LLC}, $H^i_m(M)\cong H^i_n(M)$ as $R$-modules. Now we compare the Matlis dual of $H^i_n(M)$ as a $T$-module with the Matlis dual of $H^i_m(M)$ as an $R$-module:
    \begin{align}
        &\Hom_R(H^i_m(M),E_R(k))\\
        \cong{} &\Hom_R(H^i_n(M)\otimes_T T,E_R(k))\\
        \cong{} &\Hom_T(H^i_n(M),\Hom_R(T,E_R(k)))\\
        \cong{} &\Hom_T(H^i_n(M),E_T(k)).
    \end{align}
    Here the first isomorphism uses $H^i_m(M)\cong H^i_n(M)$ and the identity $H^i_n(M)\cong H^i_n(M)\otimes_T T$, the second is hom-tensor adjunction, and the last uses
    \[
        \Hom_R(T,E_R(k))\cong E_T(k).
    \]
    Indeed,
    \[
        k\subseteq \Hom_R(T,E_R(k))\subseteq E_R(k)
    \]
    as $T$-modules, and the first inclusion is an essential extension because it is a subextension of the essential extension $k\subseteq E_R(k)$. Hence, by the characterization of injective hulls, $\Hom_R(T,E_R(k))\cong E_T(k)$. The result follows from the definition of Serre depth.
\end{proof}

\end{document}